\documentclass[11pt]{article}
\usepackage{float}
\usepackage{color}
\usepackage{amssymb}
\usepackage{amsmath}
\usepackage{amsthm}
\usepackage{graphicx}
\usepackage{color}

\def\'#1{\ifx#1i{\accent"13 \i}\else{\accent"13 #1}\fi}

\usepackage{url}

\def\P{{\rm I\!P}}
\def\L{{\rm I\!L}}
\def\l{{\rm I\!l}}

\newtheorem{theorem}{Theorem}
\newtheorem{lemma}{Lemma}
\theoremstyle{definition}
\newtheorem{definition}{Definition}

\title{Pseudoachromatic and connected-pseudoachromatic indices of the complete graph
\thanks{Research supported by CONACyT-M\'exico under projects 178395,166306, and PAPIIT-M\'exico under project IN104915 and IN104609-3,. The second author is partially supported by a CONACyT-M{\' e}xico Postdoctoral Fellowship and by the National Scholarship Program of the Slovak Republic.}}

\author{M. Gabriela Araujo-Pardo \footnotemark[3] \\ \url{garaujo@matem.unam.mx}
\and Christian Rubio-Montiel \footnotemark[2] \\ \url{christian.rubio@fmph.uniba.sk}}

\begin{document}
\maketitle

\def\thefootnote{\fnsymbol{footnote}}
\footnotetext[3]{Instituto de Matem{\' a}ticas, Universidad Nacional Aut{\' o}noma de M{\' e}xico, 04510 M{\' e}xico City, Mexico.}
\footnotetext[2]{Department of Algebra, Comenius University,  84248 Bratislava, Slovakia. UMI LAFMIA 3174 CNRS at CINVESTAV-IPN, 07300 Mexico City, Mexico.}

\begin{abstract} 
A \emph{complete $k$-coloring} of a graph $G$ is a (not necessarily proper) $k$-coloring of the vertices of $G$, such that each pair of different colors appears in an edge. A complete $k$-coloring is also called \emph{connected}, if each color class induces a connected subgraph of $G$. The \emph{pseudoachromatic index} of a graph $G$, denoted by $\psi'(G)$, is the largest $k$ for which the line graph of $G$ has a complete $k$-coloring. Analogously the \emph{connected-pseudoachromatic index} of $G$, denoted by $\psi_c'(G)$, is the largest $k$ for which the line graph of $G$ has a connected and complete $k$-coloring.
\\

In this paper we study these two parameters for the complete graph $K_n$. Our main contribution is to improve the linear lower bound for the connected pseudoachromatic index given by Abrams and Berman [Australas J Combin 60 (2014), 314--324] and provide an upper bound. These two bounds prove that for any integer $n\geq 8$ the order of $\psi_c'(K_n)$ is $n^{3/2}$.
\\

Related to the pseudoachromatic index we prove that for $q$ a power of $2$ and $n=q^2+q+1$, $\psi'(K_n)$ is at least $q^3+2q-3$ which improves the bound $q^3+q$ given by Araujo, Montellano and Strausz [J Graph Theory 66 (2011), 89--97].  

\end{abstract}


\section{Introduction}
Let $G=(V(G),E(G))$ be a finite simple graph. A \emph{complete} $k$-coloring of $G$ is an assignment $\varsigma\colon V(G)\rightarrow [k]$ (where $[k]:=\{1,\dots,k\}$), such that for each pair of different colors $i,j\in [k]$ there exists an edge $xy\in E(G)$ where $x\in \varsigma^{-1}(i)$ and $y\in \varsigma^{-1}(j)$. The \emph{pseudoachromatic number} $\psi(G)$ of $G$ is the largest $k$ for which there exists a complete $k$-coloring of $G$ \cite{MR0256930}. Some interesting results on the pseudoachromatic number and the closely related notion of the achromatic number (the maximum in proper and complete colorings) can be found in \cite{MR1224703,MR0439672,MR0272662}. 

The \emph{connected-pseudoachromatic number} $\psi_c(G)$ of a connected graph $G$ is the largest $k$ for which there exists a connected and complete $k$-coloring of $G$, i.e., a complete $k$-coloring in which each color class induces a connected subgraph. The connected pseudoachromatic number $\psi_c(G)$ of a graph $G$ with components $G_1$,$G_2$,...,$G_t$ is the largest number between $\psi_c(G_1)$,$\psi_c(G_2)$,...,$\psi_c(G_t)$. Clearly, \[\psi_c(G)\leq\psi(G).\]
The previous definition is equivalent to saying that $\psi_c(G)$ is the size of the largest complete graph minor of $G$. This value is also called the \emph{Hadwiger number} of $G$. In 1958 Hadwiger conjectured (see \cite{hadwiger1958ungeloste}) that any graph $G$, $\chi(G)\leq\psi_c(G)$, where $\chi(G)$ denotes the chromatic number of $G$ as usual. Aside from its own importance the Hadwiger Conjecture is another motivation to study the previous parameters given the pseudoachromatic number bounding the Hadwiger number.  

The connected-pseudoachromatic and the pseudoachromatic numbers of the line graph $L(G)$ of a graph $G$ are also known as the \emph{connected-pseudoachromatic index } and the \emph{pseudoachromatic index} of $G$ denoted as $\psi_c'(G)$ and $\psi'(G)$ respectively.  In this paper we study these two parameters for the complete graph $K_n$. Note that any connected and complete $k$-coloring of $L(K_n)$ is an edge coloring of $K_n$ in which each edge color class induces a connected subgraph and each pair of edge color classes share at least one vertex, therefore, we will make use of this point of view. 

Note that the colorations that induce the lower bound are given using the structure of projectives planes, it is a common technique used to obtain results for the pseudoachromatic and achromatic indices of the complete graph $K_n$, for instance see \cite{MR3265137,MR3249588,MR0439672,MR2778722,MR543176,MR989126}.

Before continuing we give some known results on $\psi_c'(K_n)$ and $\psi'(K_n)$ related to our contributions. In the case of exact values of $\psi_c'(K_n)$ only the following small values of Table \ref{exact7} are known. Table \ref{exact7} appears in \cite{MR3265137}, which also proves that $\psi_c'(K_{5a+b+1})$ is at least $9a+b$.

\begin{table}[!htbp]
\begin{center}
\begin{tabular}{|c|cccccc|}
\hline \hline
$n$ & 2 & 3 & 4 & 5 & 6 & 7 \\
\hline
$\psi_c'(K_n)$ & 1 & 3 & 4 & 6 & 7 & 10 \\
\hline
\end{tabular}
\caption{\label{exact7}Exact values for $\psi_c'(K_n)$, $2\leq n\leq 7.$}
\end{center}
\end{table}

However, in Table \ref{exact13}  (which appears also in \cite{MR3249588}) the first pseudoachromatic index values of the complete graph are listed.

\begin{table}[!htbp]
\begin{center}
\begin{tabular}{|c|cccccccccccc|}
\hline \hline
$n$ & 2 & 3 & 4 & 5 & 6 & 7 & 8 & 9 & 10 & 11 & 12 & 13 \\
\hline
$\psi'(K_n)$ & 1 & 3 & 4 & 7 & 8 & 11 & 14 & 18 & 22 & 27 & 32 & 39 \\
\hline
\end{tabular}
\caption{\label{exact13}Exact values for $\psi'(K_n)$, $2\leq n\leq 13.$}
\end{center}
\end{table}
The pseudoachromatic index for a few classes of complete graphs have been found so far, namely, Bouchet proved in \cite{MR543176} that for $n = q^2 + q + 1$ and $q$ an odd integer $\alpha'(n) = qn$, if and only if the projective plane of order $q$ exists. Consequently, it is known $\psi'(K_{q^2+q+1})$ for any odd prime power $q$ (see \cite{MR2778722,MR543176}), however, when $q$ is a power of $2$ the pseudoachromatic index is bounding by $\psi'(K_{q^2+q+1})\geq q^3+q$ (see \cite{MR2778722}) and attains exact values for $K_{q^2+2q+1-a}$ when $a\in\{-1,0,\dots,\left\lceil\frac{1+\sqrt{4q+9}}{2}\right\rceil -1\}$ (see \cite{MR,MR3249588,MR2778722}). Finally, it is proved that $\psi'(K_n)$ grows asymptotically like $n^{3/2}$ (see \cite{MR989126}).

This paper is organized as follows: In the first section we give some definitions and lemmas related to projective planes and colorings, which will be used in the second and third sections in order to give lower bounds for the pseudoachromatic connected and pseudoachromatic indices of complete graphs. In Section 2 we prove that for any integer $n\geq 8$, $\psi_c'(K_n)$ has order $n^{3/2}.$ The proof is divided in the following results:  

\begin{theorem}\label{teo1}
If $n\geq 8$ then \[\psi'_{c}(K_{n})\leq\left\lfloor \max\left\{ \min\{f_n(x),g_n(x)\} \textrm{ with } x\in\mathbb{N}\right\}\right\rfloor\] where $f_n(x)=n(n-1)/2x$ and $g_n(x)=(x+1)(n-x-1/2)$.
\end{theorem}

\begin{theorem}\label{teo2}
If $n\geq 8$ then \[\psi'_{c}(K_{n})\leq \frac{1}{\sqrt{2}}n^{\frac{3}{2}}+\Theta(n).\]
\end{theorem}

\begin{theorem}\label{teo3}
If $q$ is a prime power and $n=q^2+q+1$ then \[\left\lceil \frac{q}{2} \right\rceil n\leq \psi'_c(K_n).\]
\end{theorem}

\begin{theorem}\label{teo4}
If $n\geq 8$ then \[\frac{1}{2}n^{\frac{3}{2}}+\Theta(n) \leq \psi'_{c}(K_{n}).\]
\end{theorem}

Comparing the results of Theorems \ref{teo1} and \ref{teo2} with the upper bounds for the pseudoachromatic number given in \cite{MR,MR3249588,MR2778722,MR989126} we have that $\psi'_c(K_n)<\psi'(K_n)$ since $\psi'(K_n)=n^{\frac{3}{2}}+\Theta(n)$. Furthermore, Theorems \ref{teo3} and \ref{teo4} improve the linear lower bound that appears in \cite{MR3265137} and verifies the Hadwiger Conjecture for the line graph of the complete graphs. Finally, Theorems \ref{teo2} and \ref{teo4} show that $\psi'_c(K_n)$ grows asymptotically like $\Theta(n^{3/2})$.

In the third section we study the pseudoachromatic index $\psi'(K_n)$ of the complete graph $K_n$ for $n=q^2+q+1$ and $q$ a power of $2$. We improve the lower bound that appears in \cite{MR2778722} proving the following theorem: 

\begin{theorem}\label{teo5}
If $q$ is a power of $2$ and $n=q^2+q+1$ then \[ q^3+2q-3\leq \psi'(K_n).\]
\end{theorem}


\section{Preliminaries}


A \emph{projective plane} consists of a set of $n$ points, a set of lines, and an incidence relation between points and lines having the following properties:
\begin{enumerate} \item Given any two distinct points there is exactly one line incident with both.
\item Given any two distinct lines there is exactly one point incident with both.
\item There are four points, such that no line is incident with more than two of them. \end{enumerate}
Such a plane has $n= q^2 + q+ 1$ points (for some number $q$) and $n$ lines, each line contains $q+1$ points and each point belongs to $q+1$ lines. The number $q$ is called the \emph{order} of the projective plane. A projective plane of order $q$ is called $\Pi_q$. If $q$ is a prime power there exists $\Pi_q$, which is called the \emph{algebraic projective plane} since it arises from finite fields (see \cite{MR554919}).

Let $\P$ be the set of points of $\Pi_q$ and let $\L= \{\l_1,\dots, \l_n\}$ the set of lines of $\Pi_q$. Now identify the points of $\Pi_q$ with the set of vertices of the complete graph $K_n$.  In a natural way the set of points of each line of $\Pi_q$  induces in $K_n$ a subgraph isomorphic to $K_{q+1}$.  For each line $\l_i\in \L$, let $l_i = (V(l_i), E(l_i))$ be the subgraph of $K_n$ induced by the set of $q+1$ points of $\l_i$.  By the properties of the projective plane, for every pair $\{i,j\}\subseteq \{1,\dots,n\}$, $|V(l_i)\cap V(l_j)| =1$ and $\{E(l_1), \dots, E(l_n)\}$ is a partition of $E(K_n)$.  This way, when we say that a graph $G$ isomorphic to $K_n$ is a \emph{representation of the projective plane} $\Pi_q$ it is understood that $V(G)$ is identified with the  points of $\Pi_q$ and that there is  a family of subgraphs (lines) $\{l_1, \dots, l_n\}$ of $G$, such that for each  line $\l_i$  of $\Pi_q$, $l_i$ is the induced subgraph by the set of points  of $\l_i$. 

Given two graphs $G$ and  $H$,  the \emph{join} $G\oplus H$, is defined as the graph with vertex set $V(G)\cup V(H)$ and edge set $E(G)\cup E(H)\cup (V(G) \times V(H))$ (to abbreviate we write  the set of edges $V(G) \times V(H)$ as $V(G)V(H)$-edges). Given $S\subseteq V(G)$, $G\setminus S$ is the subgraph of $G$ induced by $V(G)\setminus S$.  

Given an edge-coloring $\varsigma\colon E(K_n)\rightarrow [k]$ we say that a vertex $u\in V(K_n)$ is an \emph{owner} of a set of colors $\mathcal{C}\subseteq [k]$  whenever for every $i\in \mathcal{C}$ there is $v\in V(K_n)$, such that $\varsigma(uv) = i$; and given a subgraph $G$ of $K_n$ we will say that $G$ is an \emph{owner} of a set of colors   $\mathcal{C}\subseteq [k]$, if each vertex of $G$ is an owner of $\mathcal{C}$. By this definition $\varsigma$ is a complete edge-coloring, if for every pair of colors in $[k]$ there is a vertex in $K_n$ which is an owner of both colors.

\begin{lemma}\label{lema2} Let  $n= q^2 + q+ 1$ with $q$ a natural number, such that $ \Pi_q $ exists. Let $K_n$ be a representation of $\Pi_q$ and let $\varsigma\colon E(K_{n})\rightarrow [k]$ be an edge-coloring of $K_{n}$. Suppose that each line $l_i$  of $K_n$ is an owner of a set of colors $\mathcal{C}_i\subseteq [k]$. Then for every pair of colors $\{c_1, c_2\}\subseteq \bigcup\limits_{i=1}^{n} \mathcal{C}_i$ there is $u\in V(G)$, which is an owner of $c_1$ and $c_2$.\end{lemma}
\begin{proof} Let $\{c_1, c_2\}\subseteq \bigcup\limits_{i=1}^{n} \mathcal{C}_i$. If there is an $i\in \{1,\dots,n\}$, such that $\{c_1, c_2\}\subseteq \mathcal{C}_i$, since $l_i$ is an owner of $\mathcal{C}_i$  it follows that each $u\in V(l_i)$ is an owner of $c_1$ and $c_2$. If $c_1\in \mathcal{C}_i$ and $c_2\in\ \mathcal{C}_j$ with $i\not=j$ there is $u\in V(G)$, such that $u= V(l_i)\cap V(l_j)$ and since $l_i$ and $l_j$ are owners of $\mathcal{C}_i$ and $\mathcal{C}_j$ respectively $u$ is an owner of $c_1$ and $c_2$.\end{proof}
 
Let us recall that any complete graph of even order $q+1$ admits a $1$-factorization $\cal M$ by $q$ perfect matchings and that any complete graph of odd order $q+1$ admits a $ 2 $-factorization $\cal H$ by $q/2$ Hamiltonian cycles (see \cite{MR0411988}). Moreover, by symmetry, any perfect matching $M$ of a complete graph of order even is an element of some 1-factorization.

Given the above, we define different edge-colorations for some special graphs. The first two will be used to give the results stated in Theorem {\ref{teo3}}, whereas Definitions \ref{definition4} and \ref{definition6} are necessary to prove the results of Theorem {\ref {teo5}}. Definitions \ref{definition3} and \ref{definition5} are technical and auxiliary and they are not used in the proofs.

\begin{definition} Let $q+1$ be an odd integer. An edge-coloring $\varsigma\colon E(K_{q+1}) \rightarrow [q/2]$ is of \emph{Type} $\cal H$, if for every $i\in [q/2]$ the set $\{ uv\in E(K_{q+1}) \colon \varsigma(uv) = i\}$ is a Hamiltonian cycle of $K_{q+1}$. \end{definition}

\begin{definition} Let $q+1$ be an even integer. An edge-coloring $\varsigma\colon E(K_{q+1}) \rightarrow [(q+1)/2]$ is of \emph{Type} $\cal P$, if we obtain $\varsigma$ in the following way: Let $G$ be the graph (isomorphic to $K_{q+2}$) obtained adding a new vertex $u$  and all the $uV(K_{q+1})$-edges. Let $\varsigma_1$ be an edge-coloring of Type $\cal H$ of $G$, and let $\varsigma(e) := \varsigma_1(e)$ for every $e\in E(K_{q+1})$. Note that the set $\{ vw\in E(K_{q+1}) \colon \varsigma(vw) = i\}$ is a Hamiltonian path of $K_{q+1}$ for every $i\in [(q+1)/2]$.\end{definition}

\begin{definition}\label{definition3}(Auxiliary definition) Let $q+1$ be an odd integer. An edge-coloring $\varsigma\colon E(K_{q+2}) \rightarrow [q+1]$ is of \emph{Type} $\cal M$, if for every $i\in [q+1]$ the set of edges $\varsigma^{-1}(i)$ is a perfect matching of $K_{q+2}$. \end{definition}

\begin{definition}\label{definition4} Let $q+1$ be an odd integer. An edge-coloring $\varsigma \colon E(K_{q+1}-uv) \rightarrow [q-1]$ is of \emph{Type} $\cal C$, if we obtain $\varsigma$ in the following way: Let $G$ be the graph (isomorphic to $K_{q}$) obtained deleting the vertex $u$ and all the $uV(K_{q})$-edges. Let $\varsigma_1$ be an edge-coloring of Type $\cal M$ of $G$, for every $e\in E(K_{q+1}-uv)$ we define $\varsigma(e) := \varsigma_1(e)$ if $ e\in E(G) $ and $ \varsigma(uw) := \varsigma_1(vw)$ for all $ w\in V(G)-v $. The edge $uv$ is called the \emph{special edge} of $K_{q+1}$ (see Figure \ref{Fig2}). \end{definition}

\begin{figure}[!htbp]
\begin{center}
\includegraphics{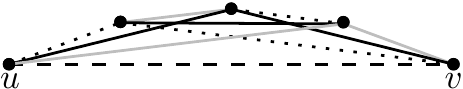}
\caption{An edge-coloring of Type $\cal C$ for $K_5-uv$. The special edge $uv$ is dashed.}\label{Fig2}
\end{center}
\end{figure}

\begin{definition}\label{definition5}(Auxiliary definition) Let $q+1$ be an odd integer, and let $M$ be a maximum matching of $K_{q+1}$. We give an edge-coloring $\varsigma : E(K_{q+1}-M) \rightarrow [q]$ of \emph{Type 1}, as follows: Let $G$ be the graph (isomorphic to $K_{q+2}$) obtained from $K_{q+1}-M$ adding a new vertex $u$, the edges of $M$ and all the $uV(K_{q+1}-M)$-edges. Let $\varsigma_1\colon E(G)\rightarrow [q+1]$ be an edge-coloring of Type $\cal M$ of $G$, such that $M\cup \{uv\}$ is an edge-color class, where $v$ is the vertex of $K_{q+1}$ which is not a vertex of $M$. For every $e\in E(K_{q+1}-M)$, $\varsigma(e) := \varsigma_1(e)$ (see Figure \ref{Fig3}).\end{definition}

By the construction of edge-coloring of Type 1, the vertex $v$ of $K_{q+1}-M$ is an owner of $q$ colors, but the remaining vertices are owners of only $q-1$ colors. Moreover, all $u_i\in K_{q+1}-M$ and $u_i\not= v$ has only one color for which it is not an owner, we call it {\emph{missing color of $u_i$}}. The following edge-coloring makes $K_{q+1}$ an owner of $q$ colors.

\begin{definition}\label{definition6} Let $q+1$ be an odd integer, $K_{q+1}$ the complete graph of order $q+1$ with the set of vertices $\{v,u_1,u_1',\dots,u_{q/2},u_{q/2}'\}$, $M=\{u_1u_1',\dots,u_{q/2}u_{q/2}\}$ a maximum matching of $K_{q+1}$, and let $G$ be the graph obtained adding a new vertex $u$ and the set of edges $\{uu_1,\dots,uu_{q/2}\}$. An edge-coloring $\varsigma : E(G) \rightarrow [q]$ is of \emph{Type 2}, if is obtaining in the following way: Let $\varsigma_1\colon E(K_{q+1}-M)\rightarrow [q]$ be an edge-coloring of Type 1 of the subgraph $K_{q+1}-M$ of $G$. For every $e\in E(K_{q+1})-M$, $\varsigma(e) := \varsigma_1(e)$. If $c_i$ is the missing color of $u_i$ and $c_i'$ is the missing color of $u_i'$ (for all $i\in[q/2]$) let $\varsigma(uu_i)=c_i$ and $\varsigma(u_iu_i')=c_i'$. The vertex $u$ is called the \emph{special vertex} of $G$ (see Figure \ref{Fig3}).\end{definition}

\begin{figure}[!htbp]
\begin{center}
\includegraphics{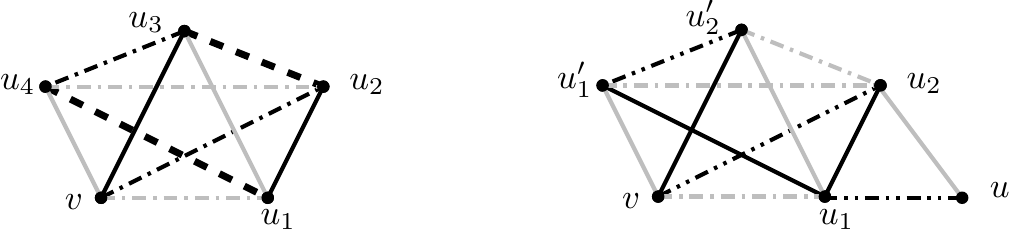}
\caption{Left: An edge-coloring of Type 1 for $K_5-M$, the maximum matching $M=\{u_1u_4,u_2u_3\}$ is dashed. Right: An edge-coloring of Type 2 for $G$.}\label{Fig3}
\end{center}
\end{figure}

\section{Connected-pseudoachromatic index of $K_n$}
We hereby begin by proving Theorems \ref{teo1} and \ref{teo2}, these results provide an upper bound of the connected-pseudoachromatic index of $K_n$. The proofs are analytic and their main idea is to compare two functions, one of which determines the maximum possible number of chromatic classes of size $x$, and the other one determines how many chromatic classes can be incident to one of the chromatic classes of size $x$.

\begin{proof}[Proof of Theorem \ref{teo1}]
Let $n\geq 8$ and  $\varsigma\colon E(K_n) \rightarrow [k]$ be a connected and complete $k$-edge-coloring of $K_n$ with $k=\psi_c'(K_n)$.  Let $x= min\{ \left|\varsigma^{-1}(i)\right| : i\in \left[k\right]\}$, that is, $x$ being the cardinality of the smallest chromatic class of $\varsigma$ considering that $x= \left|\varsigma^{-1}(k)\right|$. Since $\varsigma$ defines a partition of $E(K_n)$ it follows that $k\leq f_n(x):=n(n-1)/2x$. 

Additionally, since $\varsigma$ is connected, we can suppose that $\varsigma^{-1}(k) $ induces a tree. Moreover, the number of monochromatic trees in the subgraph induced by $E(K_{x+1})\setminus \varsigma^{-1}(k)$ is at most $\frac{\binom{x+1}{2}-x}{x}$. On the other hand, there are $(x+1)(n-(x+1))$ edges incident to some vertex of $\varsigma^{-1}(k)$ exactly once, we denote this set of edges by	 $X$. Since $\varsigma$ is complete, for every $j\in [k-1]$ there is an edge $e_j\in \varsigma^{-1}(j)$ that shares a vertex with some edge in $\varsigma^{-1}(k)$, therefore, the number of chromatic classes incident to $\varsigma^{-1}(k)$ containing some edge in $X$ is at most $(x+1)(n-(x+1))$ and the number of chromatic classes incident to $\varsigma^{-1}(k)$ containing no edge in $X$ is at most $\frac{\binom{x+1}{2}-x}{x}$. Hence, there are at most $g_n(x)-1$ chromatic classes incident with some edge in $\varsigma^{-1}(k)$ where: \[g_n(x)-1:=(x+1)(n-(x+1))+\frac{\binom{x+1}{2}-x}{x},\]
therefore: \[g_n(x)-1=(x+1)(n-x-1)+\frac{x+1}{2}-1,\]
it follows that $k\leq g_n(x)=(x+1)(n-x-1/2)$. Consequently, we have:  \[\psi'_{c}(K_{n})\leq\min\{f_n(x),g_n(x)\}.\]
Finally, we conclude that:
\[\psi'_{c}(K_{n})\leq\left\lfloor \max\left\{ \min\{f_n(x),g_n(x)\} \textrm{ with } x\in\mathbb{N}\right\}\right\rfloor.\]
\end{proof}

In order to prove Theorem \ref{teo2} we have to prove the following technical lemma:

\begin{lemma}\label{lema1}
Let $x_0,x_1\in \mathbb{R}^{+}$, such that $f_n(x_0)=g_n(x_0)$, and $f_n(x_1)=g_n(x_1)$ and suppose that: $x_0<x_1$, then \[g_n(x_0)=f_n(x_0)=\max\left\{ \min\{f_n(x),g_n(x)\} \textrm{ with } x\in\mathbb{R}^{+}\right\}\]
Furthermore: $x_0=\sqrt{n/2+1/16}-1/4$.
\end{lemma}
\begin{proof}
As $f_n$ is a hyperbola and $g_n$ is a parabola (see Figure \ref{Fig1}) we see that $f_n(x) \leq g_n(x)$ for $x_0 \leq x \leq x_1$ and $g_n(x)<f_n(x)$ in any other case when $x \in \mathbb{R}^{+}$. Since $f_n(x_0)>f_n(x_1)$ it follows that $g_n(x_0)=f_n(x_0)=\max\left\{ \min\{f_n(x),g_n(x)\} \textrm{ with } x\in\mathbb{R}^{+}\right\}.$

\begin{figure}[!htbp]
\begin{center}
\includegraphics{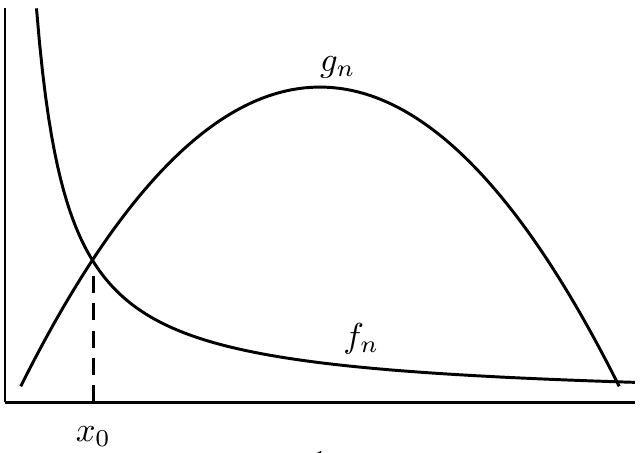}
\caption{The functions $g_n$ and $f_n$ for some fixed value $n$.}\label{Fig1}
\end{center}
\end{figure}
Finally, the equation $f_n(x_0)=g_n(x_0)$ is reduced to $n^2-(2x_0^2+2x_0+1)n+x_0(x_0+1)(2x_0+1)=0$. The positive solution for $n$ is $n=2x_0^2+x_0$ and the lemma holds true because $x_0=\sqrt{n/2+1/16}-1/4$ provides the positive solution. 
\end{proof}

\begin{proof}[Proof of Theorem \ref{teo2}]
Since $g_n(x_0)=nx_0-x_0^2-3x_0/2+n-1/2$ and $x_0=\sqrt{n/2+1/16}-1/4$ we obtain: 
\[g_n(x_0)=(n-1)(\sqrt{n/2+1/16}+1/4),\]
therefore: $g_n(x_0)\leq \frac{1}{\sqrt{2}}n^{\frac{3}{2}}+\Theta(n)$. By Lemma \ref{lema1} and Theorem \ref{teo1} the result follows.
\end{proof}


To conclude this section we prove Theorem \ref{teo3}, in which we give a lower bound for the connected-pseudoachromatic index of complete graphs $K_n$ when $n$ has a specific value related to the existence of the projective planes. 

\begin{proof}[Proof of Theorem \ref{teo3}]
Suppose $q$ is the order of a projective plane $\Pi_q$. Then $\Pi_q$ has $n=q^2+q+1$ points and $n$ lines. Let $K_n$ be a representation of $\Pi_q$ given in Section 2. We assign a coloring of $K_n$ in the following way: if $q$ is odd, each line $K^i_{q+1}$ has a coloring $\varsigma_i\colon E(K^i_{q+1})\rightarrow \{1+(i-1)q/2,\dots,q/2+(i-1)q/2\}$ of Type $\cal H$ (for $i\in [n]$). Each line $K^i_{q+1}$ may thus be divided into $q/2$ color classes, each one of size $q+1$ and it is an owner of $q/2$ colors. By Lemma \ref{lema2} the $(qn/2)$-edge-coloring is complete and it is connected by construction. Similarly, if $q$ is even, $K_n$ has a connected and complete $(q+1)n/2$-edge-coloring using edge-coloring of Type $\cal P$ in its lines.
\end{proof}

Before proving Theorem \ref{teo4} we state the following simple lemma:

\begin{lemma} \label{lema3}
For any connected graph $G$, if $H$ is a connected subgraph of $G$, then $\psi_c(G)\geq\psi_c(H)$.
\begin{proof}
Given a connected and complete coloring $\varsigma$ of $H$ with $\psi_c(H)$ colors we extend this to a coloring of $G$ by coloring the set of vertices of a component $G'$ of $G\setminus H$ with some color that appears on a vertex of $H$ adjacent to a vertex of $G'$.
\end{proof}
\end{lemma}

\begin{proof}[Proof of Theorem \ref{teo4}]
The proof depends on a strengthened version of Bertrand's ``Postulate,'' which follows from the Prime Number Theorem (see \cite{MR506522,MR0258720}): For any $\epsilon>0$, there exists an $N_\epsilon$, such that for any real $x\geq N_\epsilon$ there is a prime $q$ between $x$ and $(1+\epsilon)x$. Let $\epsilon>0$ be given, and suppose $n>(N_\epsilon+1)^2(1+\epsilon)^2$. Let $x=\sqrt{n}/(1+\epsilon) -1$, so $x\geq N_\epsilon$. We may now select a prime $q$ with $x\leq q \leq (1+\epsilon)x$. Note that $q^2+q+1\leq (x+1)^2(1+\epsilon)^2=n$. Since projective planes of all prime orders exist it follows from Theorem \ref{teo3} and Lemma \ref{lema3} that: \[\psi'_c(K_n)\geq\psi'_c(K_{q^2+q+1})\geq \frac{q}{2}(q^2+q+1)>\frac{q^3}{2}\geq \frac{x^3}{2}=\frac{(\sqrt{n}-1-\epsilon)^3}{2(1+\epsilon)^3}.\]
Since $\epsilon$ was arbitrarily small the result follows.
\end{proof}

\section{Pseudoachromatic index of complete graphs related to projective planes of even order}

To prove Theorem \ref{teo5} we will exhibit a complete $(q^3+2q-3)$-coloring in the edges of $K_{q^2+q+1}$ for $q$ a power of two since the projective plane of order $q$ always exists. To achieve this we color the edges of the complete subgraphs corresponding to the lines of the projective plane using the colorations of Type $\cal C$ and Type $2$ in such a way that each line is always an owner of a set of colors according to the coloration type. 

\begin{proof}[Proof of Theorem \ref{teo5}]
In order to prove this theorem we exhibit a complete edge-coloring of $K_{n}$ for $n=q^2+q+1$, with $k:=q^3 + 2q -3$ colors, and $q\geq 2$ a power of $2$. Let $C$ be a set of $k$ colors, and let $\{C_1, C_2,\dots, C_{n}\}$ be a partition of $C$ in the following way: $C_i$ is a set of $q-1$ colors for $1 \leq i \leq q^2-q+3$, and $C_i$ is a set of $q$ colors, for $q^2-q+4 \leq i \leq n $. Thus, we have $(q-1)(q^2-q+3)+q(2q-2)=q^3 + 2q -3$ colors. 

Let $G$ be a graph isomorphic to $K_n$ representing $\Pi_q$, and let $L$ be the set of lines of $G$. Recall that every element of $L$ is a subgraph $K_{q+1}$, but in order to simplify the writing of the proof we refer to it as a line. Furthermore, let $l_n\in L$, such that $V(l_n)=\{v_1,v_2,\dots,v_{q+1}\}$ and naming the set of lines through the vertex $v_i$ by $L_i$, such that $L_i\subseteq L\setminus\{l_{n}\}$. Every set $L_i$ has $q$ lines, and we label its lines in the following way: For each $i\in[q+1]$, let $L_{i} = \{l_{q(i-1)+1},l_{q(i-1)+2},\dots,l_{q(i-1)+q}\}$. Therefore, $L=L_1\cup L_2 \cup\dots\cup L_{q+1}\cup \{l_n\}=\{l_1,l_2,\dots,l_n\}$.

To color the edges of $K_{n}$ with $\varsigma\colon E(K_n)\rightarrow [k]$ we define a partial coloring to each line as follows:

\noindent {\it i)} Firstly, we color the lines of $L_1$, $L_2$, ... and $L_{q-1}$ using partial colorings of Type $\cal C$ and the colors of $C_{1}$, $C_{2}$ ,..., $C_{q^2-q}$; therefore, it is necessary to choose carefully the special edge in every line of $L_1\cup L_2\cup\dots L_{q-1}$. To choose these edges in lines of $L_1$ we will use $v_1$ and vertices in two lines of $L_q$, similarly, to choose the special edges in lines of $L_2$ we will use $v_2$ and vertices in two lines of $L_q$ different than the previous ones, and we continue  until choosing the special edges in the lines of $L_{q/2}$ using $v_{q/2}$ and the last two lines of $L_q$. Analogously, to choose the special edges in the lines of $L_{q/2+1}$, $L_{q/2+2}$,... and $L_{q-1}$ we use a vertex of $\{v_{q/2+1}, v_{q/2+2},..., v_{q-1}\}$ and vertices in two different lines of $L_{q+1}\setminus\{l_{n-1},l_{n-2}\}$. 

For $i\in\{1,\dots, q-1\}$ and $j\in\{1,\dots, q\}$, let $\varsigma_{q(i-1)+j}: E(l_{q(i-1)+j}-e_{q(i-1)+j}) \rightarrow C_{q(i-1)+j}$ be an edge-coloring of Type $\cal C$ with the special edge $e_{q(i-1)+j}:=v_iu_{q(i-1)+j}$ where $u_{q(i-1)+j}$ depends of the values of $i$ and $j$: 

\[u_{q(i-1)+j}:=  \left\{ \begin{array}{ll}l_{q(i-1)+j}\cap l_{q(q-1)+2i-1} &\mbox{ if } i,j\in\{1,\dots,q/2\}
\\[2ex]
l_{q(i-1)+j}\cap l_{q(q-1)+2i} &\mbox{ if  $i\in\{1,\dots,q/2\}$ and  
$ j\in\{q/2+1,\dots,q\}$ }
\\[2ex]
l_{q(i-1)+j}\cap l_{q^2+2i-1} &\mbox{ if $i\in\{q/2+1,\dots,q-1\}$ and $j\in\{1,\dots,q/2\}$  } 
\\[2ex]
l_{q(i-1)+j}\cap l_{q^2+2i} &\mbox{ if $i\in\{q/2+1,\dots,q-1\}$ and $j\in\{q/2+1,\dots,q\}$.}
 \end{array}\right.\]
 
 Figure \ref{Fig4} depicts the first and the second cases for $i=1$. Note that each line of $L_1$, $L_2$,... and $L_{q-1}$ is an owner of $q-1$ colors.

\begin{figure}[!htbp]
\begin{center}
\includegraphics{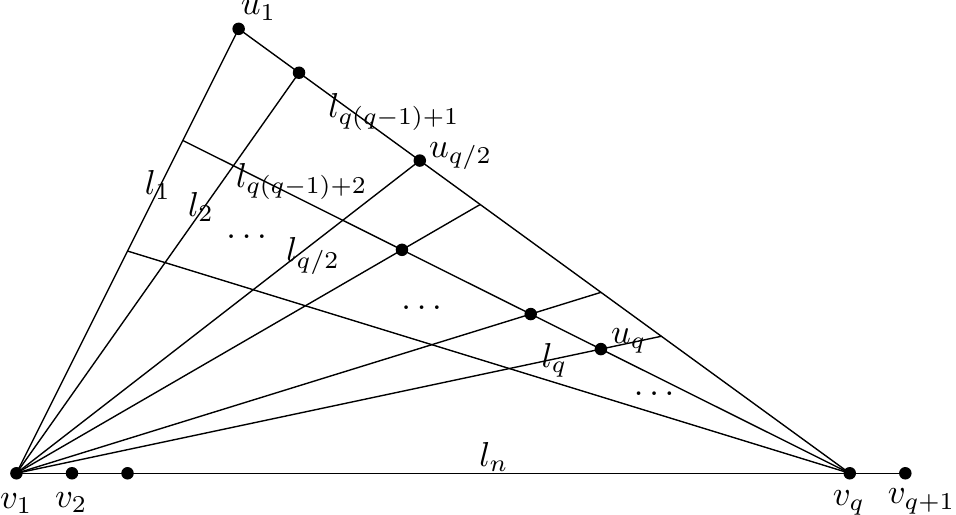}
\caption{A drawing of some lines $l_j$ and some vertices $u_j$ for $i=1$.}\label{Fig4}
\end{center}
\end{figure}

\noindent {\it ii)} Secondly, we color the lines $l_n$, $l_{n-1}$ and $l_{n-2}$ with $q-1$ colors, each line using partial colorings of Type $\cal C$, and the colors of $C_{q^2-q+3}$, $C_{q^2-q+2}$ and $C_{q^2-q+1}$; therefore, we have to choose a special edge in every line:

For $i\in\{0,1,2\}$, let $\varsigma_{n-i}: E(l_{n-i}-e_{n-i})\rightarrow C_{q^2-q+3-i}$ an edge-coloring of Type $\cal C$ with the special edge $e_{n-i}$, such that it is any edge of $l_{n-i}$. Note that $l_n$, $l_{n-1}$ and $l_{n-1}$ are owners of $q-1$ colors.

As a remark, this step is due to the fact that every color class will have $q/2+1$ edges and $(q/2+1)(q^3+2q-3)=\binom{n}{2}-3$.

\noindent {\it iii)} Now we shall assign colors to the special edges and the edges of lines in $L_{q}\cup L_{q+1}\setminus\{l_{n-1},l_{n-2}\}$. To achieve this we define subgraphs $G_j$ as uncolored lines of $L_{q}\cup L_{q+1}\setminus\{l_{n-1},l_{n-2}\}$ together with some star subgraph made of special edges, then we use partial colorings of Type $2$ and the colors of $C_{q^2-q+4}$, $C_{q^2-q+5}$,... and $C_{n}$.

Let $G_j:=(V(G_j),E(G_j))$ be graphs defined as follows:
\begin{itemize}
\item First, we use the lines of $L_{q}$. Let $V(G_j)=V(l_{q(q-1)+j})\cup\{v_{\left\lceil j/2\right\rceil }\}$ and $E(G_j)=E(l_{q(q-1)+j})\cup\{e_{(j-1)q/2+1},\dots,e_{(j-1)q/2+q/2}\}$ with $j\in\{1,\dots,q\}$.

\item Second, we use the lines of $L_{q+1}\setminus\{l_{n-1},l_{n-2}\}$. Let $V(G_j)=V(l_{q^2+j})\cup\{v_{q/2+\left\lceil j/2\right\rceil }\}$ and $E(G_j)=E(l_{q^2+j})\cup\{e_{q^2/2+(j-1)q/2+1},\dots,e_{q^2/2+(j-1)q/2+q/2}\}$ with $j\in\{1,\dots,q-2\}$.
\end{itemize}

Now we can color the $G_j$ subgraphs.

\begin{itemize}
\item Let $\varsigma_{q(q-1)+j}: E(G_j)\rightarrow C_{q^2-q+3+j}$ be an edge-coloring of Type $2$ with the special vertex $v_{\left\lceil j/2\right\rceil }$ with $j\in\{1,\dots,q\}$.

\item Let $\varsigma_{q^2+j}: E(G_j)\rightarrow C_{q^2+3+j}$ be also an edge-coloring of Type $2$ with the special vertex $v_{q/2+\left\lceil j/2\right\rceil }$ with $j\in\{1,\dots,q-2\}$.
\end{itemize}

Naturally, each line of $L_q$ and $L_{q+1}\setminus\{l_{n-1},l_{n-2}\}$ is an owner of $q$ colors.

\noindent {\it iv)} To sum up, we have already assigned a color to each edge of $ E(K_n)\backslash \{e_n,e_{n-1},e_{n-2}\}$. Finally we assign the color $1$ to the set of $ 3 $ edges $ \{e_n,e_{n-1},e_{n-2}\} $.

Since each line $l_i$ is an owner of each color of $C_i$ ($1\leq i\leq n$) by Lemma \ref{lema2} it follows that the resultant edge-coloring $ \varsigma:=\overset{n}{\underset{i=1}{\bigcup}}\varsigma_{i} $ of $ K_n $ is a complete edge-coloring using $k$ colors.
\end{proof}



The authors would like to state that after submitting this paper for review they received an e-mail from David Wood noting that there exists a result regarding the Hadwiger number of $L(G)$ relevant to this paper. DeVos, Dvorak, Fox, McDonald, Mohar and Scheide \cite{devos2014minimum} proved that the line graph of any simple graph $G$ with average degree $d$ has a clique minor of order at least $cd^{3/2}$ for some absolute constant $c > 0$. Our result gives the same order of magnitude when $G=K_n$ (when $d=n-1$); furthermore, we show that $\frac{1}{2}\leq c \leq \frac{1}{\sqrt{2}}$.


\subsection*{Acknowledgment}
The authors wish to thank the anonymous referees of this paper for their helpful remarks and suggestions.

\end{document}